\documentclass[12pt,reqno]{amsart}
\usepackage{amscd,amssymb,amsmath,amsthm}

\newtheorem{thm}[subsection]{Theorem}
\newtheorem{lemma}[subsection]{Lemma}
\newtheorem{pro}[subsection]{Proposition}
\newtheorem{cor}[subsection]{Corollary}

\newtheorem{rk}[subsection]{Remark}

\numberwithin{equation}{section} 

\newcommand{\G}{{\mathcal G}}

\def\cb{{\mathcal B}}

\newcommand{\w}{{\bf w}}
\newcommand{\s}{{\sigma}}

\newcommand{\bea}{\begin{eqnarray}}
\newcommand{\eea}{\end{eqnarray}}

\newcommand{\Z}{\mathbb{Z}}
\newcommand{\Q}{\mathbb{Q}}

\def\ce{{\mathcal E}}
\def\cg{{\mathcal G}}

\def\bn{{\mathbb N}}

\def\bq{{\mathbb Q}}

  \def\G{\Gamma}

\def\k{\kappa}
\def\m{\mu}

\def\s{\sigma} 
\def\t{\theta}

\def\w{\omega} \def\Om{\Omega}

\def\h{{\mathbf{h}}}

\def\xb{{\mathbf{x}}}

\begin{document}

\title[$p$-adic Potts-Bethe mapping]
{Chaotic behavior of the $p$-adic Potts-Bethe mapping II}

\author{Farrukh Mukhamedov}
\address{Farrukh Mukhamedov\\
 Department of Mathematical Sciences\\
College of Science, The United Arab Emirates University\\
P.O. Box, 15551, Al Ain\\
Abu Dhabi, UAE} \email{{\tt far75m@gmail.com} {\tt
farrukh.m@uaeu.ac.ae}}

\author{ Otabek Khakimov}
\address{ Otabek Khakimov\\
 Department of Mathematical Sciences\\
College of Science, The United Arab Emirates University\\
P.O. Box, 15551, Al Ain\\
Abu Dhabi, UAE} \email{{\tt hakimovo@mail.ru} {\tt
otabek.k@uaeu.ac.ae}}

\begin{abstract}
In our previous investigations, we have developed the
renormalization group method to $p$-adic $q$-state Potts model on the Cayley tree of order $k$.
This method is closely related to the examination of dynamical behavior of the $p$-adic Potts-Bethe
mapping which depends on parameters $q,k$.  In \cite{MFKh18} we have considered the case when $q$ is not divisible by $p$, and under some conditions
it was established that  the mapping is conjugate to the full shift.  The present paper is a continuation of the mentioned paper, but here we investigate the case when $q$ is divisible by $p$ and  $k$ is arbitrary. We are able to fully describe the dynamical behavior of the $p$-adic Potts-Bethe mapping by means of Markov partition. Moreover, the existence of Julia set is established, over which the mapping enables a chaotic behavior.   We point out that
a similar result is not known in the case of real numbers
(with rigorous proofs).
 \vskip 0.3cm \noindent {\it
Mathematics Subject Classification}: 37B05, 37B10,12J12, 39A70\\
{\it Key words}: $p$-adic numbers, Potts-Bethe mapping, attraction, repeller, chaos,
shift.
\end{abstract}

\maketitle

\section{Introduction}

The present paper is a continuation of \cite{MFKh18}, where we have started to
investigate chaotic behavior of the Potts-Bethe mapping over the $p$-adic field (here $p$ is some prime number). Note that the mapping is governed by
\begin{equation}\label{func}
f_{\theta,q,k}(x)=\left(\frac{\theta x+q-1}{x+\theta+q-2}\right)^k,
\end{equation}
where $k,q\in\mathbb N$ and $\theta\in B_1(1)$.
In the mentioned paper, we have considered the case when $q$ is not divisible by $p$, i.e. $|q|_p=1$. In that setting, under some conditions  we were able to prove that $f_{\theta,q,k}$ is conjugate to the full shift. In the current paper, we are going to study the same Potts-Bethe mapping when $q$ is divisible by $p$, i.e. $|q|_p<1$. It is known that
the thermodynamic behavior of the central site of the Potts model with nearest-neighbor interactions
on a Cayley tree is reduced to the recursive system which is given by \eqref{func}.  The existence of at least two to non-trivial $p$-adic Gibbs measures indicates that the phase transition may exist. This is closely connected to the chaotic behavior of the
associated dynamical system\cite{FTC,Monr,M13,M15}. Therefore, it is important to investigate chaotic properties of \eqref{func}.

We stress that the Potts-Ising mapping is a particular case of the Potts-Bethe mapping, which can be obtained from \eqref{func} by putting $q=2$. Recently, in \cite{MAD17,MFKhO_ising}  under some condition, a Julia set of the  Potts-Ising mapping have been
described, and it was shown this mapping is conjugate to the full shift. Therefore, it is natural to consider the
the Potts-Bethe mapping for $q\geq3$ with $|q|_p<1$ and $k\geq 2$.
In \cite{RKh} all fixed points of $f_{\theta,q,k}$ have been found when $k=2$ and $|q|_p<1$.
Then, using these fixed points in \cite{MFKh161,MFKh16} the
dynamics of \eqref{func}
whenever $k=2$ and $|q|_p<1$ has been investigated. Recently in \cite{ALS17,SA15} the Potts-Bethe mapping has been studied at $k=3$ with $|q|_p<1$. In the present paper, we are going to consider a more general case, i.e. arbitrary $k$ and $|q|_p<1$.  To formulate our main result let  us recall some necessary notions.

It is easy to notice that the function \eqref{func} is defined on $\mathbb Q_p\setminus\{x^{(\infty)}\}$,
where $x^{(\infty)}=2-q-\theta$. For the sake of convenience,
we write $Dom(f_{\theta,q,k}):=\mathbb Q_p\setminus\{x^{(\infty)}\}$.
Let us denote
$$
\mathcal P_{x^{(\infty)}}=\bigcup_{n=1}^\infty f_{\theta,q,k}^{-n}(x^{(\infty)}).
$$

On can see that that the set $\mathcal P_{x^{(\infty)}}$ is at most countable, and could be empty for some $k,q$ and
$\theta$ (see Section 3). If it is not an empty set, then for any $x_0\in\mathcal P_{x^{(\infty)}}$
there exists a $n\geq1$ such that after $n$-times we will "lost" that point.


Let $x^{(0)}$ be a fixed point of an analytic function $f$ and
$$
\lambda=\frac{d}{d x}f(x^{(0)}).
$$
The fixed point $x^{(0)}$ is called {\it attractive} if
$0<|\lambda|_p<1$, {\it indifferent} if $|\lambda|_p=1$, and {\it
repelling} if $|\lambda|_p>1$.

For an attractive point $x^{(0)}$ of the function $f$,  its the basin of attraction is defined by
$$
A(x^{(0)})=\{x\in\mathbb Q_p: \lim_{n\to\infty}f^n(x)=x^{(0)}\}
$$
where $f^n=\underbrace{f\circ f\circ\dots\circ f}_n$.

The main result of the present paper is given in the following theorem.

\begin{thm}\label{thm_mainasos}
Let $p\geq3$ and $|q|_p<1$, $|\theta-1|_p<1$. Then the dynamical structure of the system
$(\mathbb Q_p, f_{\theta,q,k})$ is described as follows:\\
\begin{enumerate}
\item[$(A)$.] If $|k|_p\leq|q+\theta-1|_p$ then $Fix(f_{\theta,q,k})=\{x_0^*\}$ and
$$
A(x_0^*)=Dom(f_{\theta,q,k}).
$$
\item[$(B)$.] Assume that
$|k|_p>|q+\theta-1|_p$ and $|\theta-1|_p<|q^2|_p$. Then there exists a non empty
set $J_{f_{\theta,q,k}}\subset Dom(f_{\theta,q,k})\setminus\mathcal P_{x^{(\infty)}}$ which is invariant w.r.t
$f_{\theta,q,k}$
and
$$
A(x_0^*)=Dom(f_{\theta,q,k})\setminus\left(\mathcal P_{x^{(\infty)}}\cup J_{f_{\theta,q,k}}\right).
$$

 Moreover,  if $\kappa_p$ is the GCF (greatest common factor) of $k$ and $p-1$,  then the followings hold:
\begin{enumerate}
\item[$(B1)$.] if $\kappa_p=1$ then there exists
$x_*\in Fix(f_{\theta,q,k})$ such that $x_*\neq x_0^*$ and $J_{f_{\theta,q,k}}=\{x_*\}$;
\item[$(B2)$.] if $\kappa_p\geq2$ then
$(J_{f_{\theta,q,k}},f_{\theta,q,k},|\cdot|_p)$ is topologically conjugate to the
full shift dynamics of $\kappa_p$ symbols.
\end{enumerate}
\end{enumerate}
Here $x_0^*=1$.
\end{thm}

\begin{rk} As we mentioned earlier that in \cite{ALS17} the chaoticity of \eqref{func} has been studied at $k=3$. To establish this result, it was 
essentially used particular properties of the fixed points of \eqref{func}. An advantage of the present paper is that we are able to prove the chaoticity of the Potts-Bethe mapping for arbitrary values of $k$, and moreover, we are not even using the existence of  the fixed points. Roughly speaking, we are constructing (explicitly) a Markov partition of the mapping \eqref{func} which allows us to prove the main result of this paper. 
\end{rk}

\begin{rk} In \cite{MR1,MR2} the authors established that the function \eqref{func} may have at least one fixed point, and moreover, it was found a necessary condition (i.e. $q$ is divisible by $p$ ($q\in p\bn$)) for the existence of more than one its fixed points. Therefore, they formulated a conjecture: \texttt{if $q\in p\bn$ and $\theta\in \ce_p$, for any $k\in\bn$, then the function \eqref{func} has at least two fixed points}. 
Our Theorem \ref{thm_mainasos} (A) shows that the mentioned conjecture is not always true.  
\end{rk}

We stress that, in the $p$-adic setting, due to lack of the convex structure of the set
of $p$-adic Gibbs measures, it was quite difficult to constitute a phase transition with
some features of the set of $p$-adic Gibbs measures. However, Theorem \ref{thm_mainasos}(B2)
yields
that the set of $p$-adic Gibbs measures is huge which is a' priori not clear (see \cite{MR2}). Moreover, the advantage of
the present work allows to find lots of periodic $p$-adic Gibbs measures for the $p$-adic Potts
model.
Besides, this theorem together with the results of \cite{MFKh162} will open new perspectives in
investigations of generalized $p$-adic self-similar sets.

On the other hand, our results shed some light to the question of investigation of dynamics of rational functions
in the $p$-adic analysis, since a global dynamical
structure of rational maps on $\mathbb{Q}_p$ remains unclear.  Some particular rational functions have been considered in \cite{B0,B1, DS,FFLW,KM1,KhN,MFSMKhO1}.

\section{Preliminaries}

\subsection{$p$-adic numbers}

Let $\Q$ be the field of rational numbers. For a fixed prime number
$p$, every rational number $x\ne 0$ can be represented in the form
$x = p^r{n\over m}$, where $r, n\in \Z$, $m$ is a positive integer,
and $n$ and $m$ are relatively prime with $p$: $(p, n) = 1$, $(p, m)
= 1$. The $p$-adic norm of $x$ is given by
$$|x|_p=\left\{\begin{array}{ll}
p^{-r}\ \ \mbox{for} \ \ x\ne 0\\
0\ \ \mbox{for} \ \ x = 0.
\end{array}\right.
$$
This norm is non-Archimedean  and satisfies the so called strong
triangle inequality
$$|x+y|_p\leq \max\{|x|_p,|y|_p\}.$$

The completion of $\Q$ with respect to the $p$-adic norm defines
the $p$-adic field
 $\Q_p$. Recall that $\mathbb Q_p$ is not ordered field. So,
 we may compare two $p$-adic numbers only w.r.t. their $p$-adic norms.

 For given $p$-adic numbers $x$ and $y$, for the sake of convenience,
 we write $x=O[y]$, $x=o[y]$ or $y=o[x]$ when
 $|x|_p=|y|_p$, $|x|_p<|y|_p$ or $|y|_p<|x|_p$, respectively. For example,
 if $x=1-p+p^2$ we can write $x=O[1]$, $x-1=o[1]$ or $x-1+p=o[p]$. In other words,
 this means that $|x|_p=1$, $|x-1|_p<1$ or $|x-1+p|_p<|p|_p$, respectively.
So, the symbols $O[\cdot]$ and $o[\cdot]$ make our work easier when we need
to calculate the $p$-adic norm of $p$-adic numbers.
It is easy to see that $y=O[x]$ if and only if $x=O[y]$. Moreover,
$y=O[x]$ implies that $y=o[\frac{x}{p}]$.

We give some basic properties of $O[\cdot]$ and $o[\cdot]$, which will be used later on.
\begin{lemma}\label{lemOo} Let $x,y\in\mathbb Q_p$. Then the following statements hold:\\
\begin{enumerate}
\item[1.] $O[x]O[y]=O[xy]$;\\
\item[2.] $xO[y]=O[y]x=O[xy]$;\\
\item[3.] $O[x]o[y]=o[xy]$;\\
\item[4.] $o[x]o[y]=o[xy]$;\\
\item[5.] $xo[y]=o[y]x=o[xy]$;\\
\item[6.] $\frac{O[x]}{O[y]}=O\left[\frac{x}{y}\right]$, if $y\neq0$;\\
\item[7.] $\frac{o[x]}{O[y]}=o\left[\frac{x}{y}\right]$, if $y\neq0$.
\end{enumerate}
\end{lemma}
 Any $p$-adic number $x\ne 0$ can be uniquely represented
in the canonical form
\begin{equation}\label{ek}
x = p^{\gamma(x)}(x_0+x_1p+x_2p^2+\dots),
\end{equation}
where $\gamma(x)\in \Z$ and the integers $x_j$ satisfy: $x_0 > 0$,
$0\leq x_j \leq p - 1$. In this case $|x|_p =
p^{-\gamma(x)}$.

For each $a\in \bq_p$, $r>0$ we denote $$ B_r(a)=\{x\in \bq_p :
|x-a|_p< r\}.$$ We recall that $\mathbb{Z}_p=\{x\in \Q_p:
|x|_p\leq 1\}$ and $\mathbb Z_p^*=\{x\in\mathbb Q_p: |x|_p=1\}$ are the set of all \textit{$p$-adic integers}
and {\it $p$-adic units}, respectively.

The following lemma is known as the Hensel's Lemma
\begin{lemma}\cite{Borevich}
Let $F(x)$ be a polynomial whose coefficients are $p$-adic integers.
Let $x^*$ be a $p$-adic integer such that for some $i\geq0$
we have 
$$
F(x^*)\equiv0(\operatorname{mod }p^{2i+1}),\ \ \ F'(x^*)\equiv0(\operatorname{mod }p^{i}),\ \ \ F'(x^*)\not\equiv0(\operatorname{mod }p^{i+1}).
$$
Then $F(x)$ has a $p$-adic integer root $x_*$ such that $x_*\equiv x^*(\operatorname{mod }p^{i+1})$.
\end{lemma}
The \textit{$p$-adic exponential} is defined by
$$\exp_p(x) =\sum^\infty_{n=0}{x^n\over n!},$$
which converges for every $x\in B_{p^{-1/(p-1)}}(0)$. Denote
$$
\mathcal E_p=\left\{x\in\mathbb Q_p: |x-1|_p<p^{-1/(p-1)}\right\}.
$$
This set is the range of the $p$-adic exponential function.
The following fact is well-known.

\begin{lemma}\cite{MFSMKhO2}\label{epproperty}
Let $p\geq3$. The set $\mathcal E_p$ has the following properties:\\
$(a)$ $\mathcal E_p$ is a group under multiplication;\\
$(b)$ $|a-b|_p<1$ for all $a,b\in\mathcal E_p$;\\
$(c)$ If $a,b\in\mathcal E_p$ then it holds
$|a+b|_p=1$.\\
$(d)$ If $a\in\mathcal E_p$, then
there is an element $h\in B_{p^{-1/(p-1)}}(0)$ such that
$a=\exp_p(h)$.
\end{lemma}

\begin{lemma}\cite{MFKh18}\label{alpbetgam}
Let $k\geq2$ and $p\geq3$. Then for any $\alpha,\beta\in\mathcal E_p$ there exists
$\gamma\in\mathcal E_p$ such that
\begin{equation}\label{abg}
\sum_{j=0}^{k-1}\alpha^{k-1-j}\beta^j=k\gamma
\end{equation}
\end{lemma}
\begin{cor}\label{coralb}
Let $p\geq3$ and $k\in\mathbb N$. Then one has
$$
\alpha^k-\beta^k=k(\alpha-\beta)+o\left[k(\alpha-\beta)\right],\ \ \ \forall\alpha,\beta\in\mathcal E_p.
$$
\end{cor}
\begin{proof} Let $\alpha,\beta\in\mathcal E_p$.
Due to Lemma \ref{alpbetgam} we have
$$
\sum_{j=0}^{k-1}\alpha^{k-1-j}\beta^j=k+k(\gamma-1),
$$
where $\gamma-1=o[1]$.

Using the last one with Lemma \ref{lemOo} we obtain
\begin{eqnarray*}
\alpha^k-\beta^k&=&(\alpha-\beta)\sum_{j=0}^{k-1}\alpha^{k-j-1}\beta^j\\
&=&k(\alpha-\beta)+k(\alpha-\beta)(\gamma-1)\\
&=&k(\alpha-\beta)+O[k(\alpha-\beta)]o[1]\\
&=&k(\alpha-\beta)+o[k(\alpha-\beta)],
\end{eqnarray*}
which is the required relation.
\end{proof}

\subsection{$p$-adic subshift}

Let $f:X\to\mathbb Q_p$ be a map from a compact
open set $X$ of $\mathbb Q_p$ into $\mathbb Q_p$. We assume that (i) $f^{-1}(X)\subset X$;
(ii) $X=\cup_{j\in I}B_{r}(a_j)$ can be written as a finite disjoint
union of balls of centers $a_j$ and of the same radius $r$ such that for each $j\in I$ there is an integer
$\tau_j\in\mathbb Z$ such that
\begin{equation}\label{tau}
|f(x)-f(y)|_p=p^{\tau_j}|x-y|_p,\ \ \ \ x,y\in B_r(a_j).
\end{equation}
For such a map $f$, define its Julia set by
\begin{equation}\label{J}
J_f=\bigcap_{n=0}^\infty f^{-n}(X).
\end{equation}
It is clear that $f^{-1}(J_f)=J_f$ and then $f(J_f)\subset J_f$. The triple $(X,J_f,f)$ is
called a $p$-adic {\it repeller} if all $\tau_j$ in \eqref{tau} are
positive.
 For any $i\in I$, we let
$$
I_i:=\left\{j\in I: B_r(a_j)\cap
f(B_r(a_i))\neq\varnothing\right\}=\{j\in I: B_r(a_j)\subset
f(B_r(a_i))\}
$$
(the second equality holds because of the expansiveness and of the
ultrametric property). Then define a matrix $A=(a_{ij})_{I\times
I}$, called \textit{incidence matrix} as follows
$$
a_{ij}=\left\{\begin{array}{ll}
1,\ \ \mbox{if }\ j\in I_i;\\
0,\ \ \mbox{if }\ j\not\in I_i.
\end{array}
\right.
$$
If $A$ is irreducible, we say that $(X,J_f,f)$ is
\textit{transitive}. Here the irreducibility of $A$  means, for any
pair $(i,j)\in I\times I$ there is positive integer $m$ such that
$a_{ij}^{(m)}>0$, where $a_{ij}^{(m)}$ is the entry of the matrix
$A^m$.

Given $I$ and the irreducible incidence matrix $A$ as above, we denote
$$
\Sigma_A=\{(x_k)_{k\geq 0}: \ x_k\in I,\  A_{x_k,x_{k+1}}=1, \ k\geq
0\}
$$
which is the corresponding subshift space, and let $\sigma$ be the
shift transformation on $\Sigma_A$. We equip $\Sigma_A$ with a
metric $d_f$ depending on the dynamics which is defined as follows.
First for $i,j\in I,\ i\neq j$ let $\k(i,j)$ be the integer such
that $|a_i-a_j|_p=p^{-\k(i,j)}$. It clear that $\k(i,j)<\log_p(r)$. By
the ultra-metric inequality, we have
$$
|x-y|_p=|a_i-a_j|_p\ \ \ i\neq j,\ \forall x\in B_r(a_i), \forall
y\in B_r(a_j)
$$
For $x=(x_0,x_1,\dots,x_n,\dots)\in\Sigma_A$ and
$y=(y_0,y_1,\dots,y_n,\dots)\in\Sigma_A$, define
$$
d_f(x,y)=\left\{\begin{array}{ll}
p^{-\tau_{x_0}-\tau_{x_1}-\cdots-\tau_{x_{n-1}}-\k(x_{n},y_{n})}&, \mbox{ if }n\neq0\\
p^{-\k(x_0,y_0)}&, \mbox{ if }n=0
\end{array}\right.
$$
where $n=n(x,y)=\min\{i\geq0: x_i\neq y_i\}$. It is clear that $d_f$
defines the same topology as the classical metric which is defined
by $d(x,y)=p^{-n(x,y)}$.

\begin{thm}\cite{FL2}\label{xit} Let $(X,J_f,f)$ be a transitive $p$-adic weak repeller with incidence matrix $A$.
Then the dynamics $(J_f,f,|\cdot|_p)$ is isometrically conjugate to
the shift dynamics $(\Sigma_A,\sigma,d_f)$.
\end{thm}

\section{{\it Proof of Theorem \ref{thm_mainasos}}: part $(A)$}

In what follows, we always assume that $p\geq3$ and $|q|_p<1$.
To prove Theorem \ref{thm_mainasos} $(A)$ we need the following
auxiliary fact.

\begin{lemma}\label{yaxshilemma}
Let $p\geq3$ and $k\in\mathbb N$. If $a\in\mathcal E_p$ and $|a-1|_p\geq|k|_p$ then
$|x^k-a|_p\geq|a-1|_p$
for
any $x\in\mathbb Q_p$.
\end{lemma}

\begin{proof} Take an arbitrary $a\in\mathcal E_p$ such that
$|a-1|_p\geq|k|_p$. Now consider several cases.\\
{\it Case $x\not\in\mathbb Z_p^*$.} Then we immediately
get $|x^k-1|_p\geq1$. From $|a-1|_p<1$, using the strong triangle inequality one has
$|x^k-a|_p\geq1$. This yields that $|x^k-a|_p>|a-1|_p$.\\
{\it Case $x\in\mathcal E_p$.} Then noting $|x-1|_p<1$, due to Corollary \ref{coralb} we obtain
$|x^k-1|_p<|k|_p$. The last one together with $|a-1|_p\geq|k|_p$ implies that
$|x^k-a|_p=|a-1|_p$.\\
{\it Case $x\in\mathbb Z_p^*\setminus\mathcal E_p$.} In this case, $x$ has the following canonical
form:
$$
x=x_0+x_1\cdot p+x_2\cdot p^2+\dots
$$
where $2\leq x_0\leq p-1$ and $0\leq x_i\leq p-1$, $i\geq1$. Then we obtain
$\frac{x}{x_0}\in\mathcal E_p$. According to Corollary \ref{coralb} one has
$$
\left(\frac{x}{x_0}\right)^k=1+O[k(x-x_0)]=1+o[k].
$$
Consequently,
$|x^k-x_0^k|_p<|k|_p$ which yields $|x^k-1|_p=|x_0^k-1|_p$.
Now we need to check two cases $|x_0^k-1|_p=1$ and $|x_0^k-1|_p<1$, separately.

Suppose that $|x_0^k-1|_p=1$. Then, due to  $|a-1|_p<1$ one has
$|x_0^k-a|_p=1$. Hence, $|x^k-a|_p>|a-1|_p$.

Let us assume that $|x_0^k-1|_p<1$. For convenience, let us  write $k=mp^s$, where $s\geq1$ and $(m,p)=1$.
Then noting $x_0^{p}\equiv x_0(\operatorname{mod }p)$, from $x^{mp^s}\equiv1(\operatorname{mod }p)$ we obtain $|x_0^m-1|_p<1$.
Thanks to Corollary \ref{coralb} one finds
$$
x_0^{mp^s}-1=p^{s}(x_0^m-1)+o\left[p^s(x_0^m-1)\right],
$$
which
yields that $|x_0^k-1|_p<|k|_p$. Hence, from $|a-1|_p\geq|k|_p$ we obtain
$|x_0^k-a|_p=|a-1|_p$. Consequently, $|x^k-a|_p=|a-1|_p$.
This completes the proof.
\end{proof}

\begin{rk}
We notice that the set $\mathcal P_{x^{(\infty)}}$ is empty if $|k|_p\leq|q+\theta-1|_p$.
Indeed, from $x^{(\infty)}\in\mathcal E_p$, where $x^{(\infty)}=2-q-\theta$ and $|x^{(\infty)}-1|_p\geq|k|_p$, due to Lemma
\ref{yaxshilemma} we infer that
$$
\left|f_\theta^n(x)-x^{(\infty)}\right|_p\geq\left|x^{(\infty)}-1\right|_p,\ \ \ \forall n\in\mathbb N,\ \forall x\in Dom(f_\theta).
$$
Hence, noting $x^{(\infty)}\neq1$ we can conclude that $\mathcal P_{x^{(\infty)}}=\emptyset$.
\end{rk}
Let us define
\begin{equation}\label{g(x)h(x)}
g_{\theta,q}(x)=\frac{\theta x+q-1}{x+q+\theta-2}
\end{equation}
We notice that $f_{\theta,q,k}(x)=\left(g_{\theta,q}(x)\right)^k$ for any $x\in Dom(f_{\theta,q,k})$.
It is clear that the function $f_{\theta,q,k}$ has a fixed point $x_0^*=1$.

\begin{proof}[Proof of Theorem \ref{thm_mainasos}: $(A)$]
Let $|k|_p\leq|q+\theta-1|_p$. Let us denote
$$
\begin{array}{ll}
K_1=\left\{x\in\mathbb Q_p: |x-1|_p<|q+\theta-1|_p\right\},\\[2mm]
K_2=\left\{x\in\mathbb Q_p: |x-1|_p=|x-2+q+\theta|_p\right\}.
\end{array}
$$
First, we show that $f_{\theta,q,k}(x)\in K_1\cup K_2$ for any $x\notin K_1\cup K_2$. Then
we prove that $f_{\theta,q,k}(x)\in K_1$ for any $x\in K_2$. Finally, we show that
$f_{\theta,q,k}^n(x)\to1$ for any $x\in K_1$.

Indeed,
let $x\notin K_1\cup K_2$.
From $|q+\theta-1|_p<1$, due to Lemma \ref{yaxshilemma}
we obtain
$$
|f_{\theta,q,k}(x)-2+q+\theta|_p\geq|q+\theta-1|_p,
$$
which is equivalent to either $|f_{\theta,q,k}(x)-1|_p<|q+\theta-1|_p$ or $|f_{\theta,q,k}(x)-1|_p=|f_{\theta,q,k}(x)-2+q+\theta|_p$.
This yields that $f_{\theta,q,k}(x)\in K_1\cup K_2$.

Now assume that $x\in K_2$. Then we have
\begin{eqnarray*}
g_{\theta,q}(x)&=&1+\frac{(\theta-1)(x-1)}{x+q+\theta-2}\nonumber\\
&=&1+(\theta-1)O[1]\nonumber\\
&=&1+O[\theta-1]\nonumber\\
&=&1+o[1]
\end{eqnarray*}
this means $g_{\theta,q}(x)\in\mathcal E_p$. Then thanks
to Corollary \ref{coralb} one gets
$$
|f_{\theta,q,k}(x)-1|_p<|k|_p.
$$
The last one together with $|k|_p\leq|q+\theta-1|_p$ implies $|f_{\theta,q,k}(x)-1|_p<|q+\theta-1|_p$, hence
$f_{\theta,q,k}(x)\in K_1$.

Finally, we suppose that $x\in K_1$. One has
\begin{eqnarray*}
g_{\theta,q}(x)&=&1+\frac{(\theta-1)(x-1)}{x+q+\theta-2}\nonumber\\
&=&1+\frac{(\theta-1)(x-1)}{O[q+\theta-1]}\nonumber\\
&=&1+(\theta-1)o[1]\nonumber\\
&=&1+o[\theta-1]\nonumber\\
&=&1+o[1]
\end{eqnarray*}
This again means $g_{\theta,q}(x)\in\mathcal E_p$. Then Corollary \ref{coralb} implies
$$
f_{\theta,q,k}(x)-1=O\left[\frac{k(\theta-1)(x-1)}{q+\theta-1}\right].
$$
Noting $|q+\theta-1|_p>|k(\theta-1)|_p$, from the last one, we obtain
$$
|f_{\theta,q,k}(x)-1|_p<|x-1|_p.
$$
Hence,
$$
|f_{\theta,q,k}^n(x)-1|_p\leq\frac{1}{p^n}|x-1|_p.
$$
which yields that $f_{\theta,q,k}^n(x)\to1$ as $n\to\infty$.

This completes the proof.
\end{proof}

\section{Proof of Theorem \ref{thm_mainasos}: part (B)}

In this section, we are going to study
the dynamics of $f_{\theta,q,k}$ when $|\theta-1|_p<|q^2|_p$ and $|q|_p<|k|_p$. In what follows, we need the following auxiliary fact.

\begin{pro}\label{pro_limits}
Let $p\geq3$ and $|\theta-1|_p<|q|_p<|k|_p$. If $x\in Dom(f_{\theta,q,k})$ with
$|x-2+q+\theta|_p>|\theta-1|_p$ then
$f_{\theta,q,k}^n(x)\to1$ as $n\to\infty$.
\end{pro}
\begin{proof}
First, we notice that $|x-2+q+\theta|_p>|\theta-1|_p$ implies $|x-1+q|_p>|\theta-1|_p$.
Due to $|\theta-1|_p<|q|_p$, we are going to consider two cases: $(i)$ $|x-1+q|_p\geq|q|_p$ and
$(ii)$ $|\theta-1|_p<|x-1+q|_p<|q|_p$, respectively.

{\it Case $(i)$.} Let $|x-1+q|_p\geq|q|_p$. This means that
either $x\in B_{|q|_p}(1)$ or $|x-1+q|_p=|x-1|_p$. So, let us consider each case one by one.
First, we show that the condition
$|x-1+q|_p=|x-1|_p$ yields $f_{\theta,q,k}(x)\in B_{|q|_p}(1)$. Then we prove
that
$f_{\theta,q,k}^n(x)\to1$ for any $x\in B_{|q|_p}(1)$.

Let us pick $x\in\mathbb Q_p$ with $|x-1+q|_p=|x-1|_p$. This yields that
$|x-1|_p\geq|q|_p$, so from $|\theta-1|_p<|q|_p$ one gets
\begin{eqnarray*}
g_{\theta,q}(x)-1&=&\frac{(\theta-1)(x-1)}{x-1+q+\theta-1}\\[2mm]
&=&\frac{(\theta-1)(x-1)}{x-1+o[q]}
=\frac{(\theta-1)(x-1)}{O[x-1]}\\[2mm]
&=&(\theta-1)O[1]
=O[\theta-1]
\end{eqnarray*}
Since $|\theta-1|_p<|q|_p$ and $|k|_p\leq1$, due to Corollary \ref{coralb} we
obtain $|f_{\theta,q,k}(x)-1|_p<|q|_p$. This means that $f_{\theta,q,k}(x)\in B_{|q|_p}(1)$.

Now let us suppose that $x\in B_{|q|_p}(1)$. Then we have
\begin{eqnarray*}
g_{\theta,q}(x)-1&=&\frac{o[q](x-1)}{q+o[q]}
=\frac{o[q](x-1)}{O[q]}\\[2mm]
&=&o[1](x-1)=o[x-1]
\end{eqnarray*}
Hence, again thanks to Corollary \ref{coralb} one has $|f_{\theta,q,k}(x)-1|_p<|x-1|_p$, which implies
$$
|f_{\theta,q,k}^n(x)-1|_p\leq\frac{1}{p^n}|x-1|_p,\ \ \ \forall n\in\mathbb N.
$$
So, $f_{\theta,q,k}^n(x)\to1$ as $n\to\infty$.

{\it Case $(ii)$} Let $|\theta-1|_p<|x-1+q|_p<|q|_p$. Then
\begin{eqnarray*}
g_{\theta,q}(x)-1&=&\frac{o[x-1+q]O[q]}{O[x-1+q]}\\[2mm]
&=&o[1]O[q]=o[q].
\end{eqnarray*}
Then again Corollary \ref{coralb} yields
$|f_{\theta,q,k}(x)-1|_p<|q|_p$. Hence, by $(i)$ we have $f_{\theta,q,k}^n(x)\to1$ as $n\to\infty$.
This completes the proof.
\end{proof}

\begin{cor}\label{corssilka}
Let $p\geq3$ and $|\theta-1|_p<|q|_p<|k|_p$. If $|x-1+q|_p\geq|q|_p$
then $f_{\theta,q,k}^n(x)\to1$ as $n\to\infty$.
\end{cor}

\begin{proof} Let $|x-1+q|_p\geq|q|_p$. Then
from $|\theta-1|_p<|q|_p$ using the strong triangle inequality
one finds $|x-2+q+\theta|_p>|\theta-1|_p$ which thanks to Proposition
\ref{pro_limits} yields $f_{\theta,q,k}^n(x)\to1$.
\end{proof}

Let us denote
$$
Sol_p(x^k-1)=\left\{\xi\in\mathbb Z_p^*: \xi^k=1\right\},\ \ \ \kappa_p=|Sol_p(x^k-1)|.
$$
By another words,  $\kappa_p$ is the number of solutions of the equation $x^k=1$ in $\bq_p$. From the results of \cite{MS13} we infer that 
$\kappa_p$ is the GCF (greatest common factor) of $k$ and $p-1$.  Therefore, 
it is clear that $1\leq\kappa_p\leq k$.

For a given $\xi_i\in Sol_p(x^k-1),\ i\in\{1,\dots,\kappa_p\}$ we define
\begin{equation}\label{x_i}
{x}_{\xi_i}=\left\{
\begin{array}{ll}
1-q+(k-1)\left(1-\frac{q}{2}+\frac{(k-2)q^2}{6k}\right)(\theta-1), & \mbox{if }\ \xi_i=1\\[3mm]
2-q-\theta+\frac{q}{1-\xi_i}(\theta-1), & \mbox{if }\ \xi_i\neq1
\end{array}\right.
\end{equation}
and
\begin{equation}\label{X}
X=\bigcup\limits_{i=1}^{\kappa_p}B_r({x}_{\xi_i}),\ \ r=|q(\theta-1)|_p.
\end{equation}

\begin{lemma} Let $p\geq3$ and $|\theta-1|_p<|q|_p<|k|_p$. If $x_{\xi_i}$ is given by
\eqref{x_i} and $r=|q(\theta-1)|_p$ then
$B_r({x}_{\xi_i})\cap B_r({x}_{\xi_j})=\emptyset$ if $i\neq j$.
\end{lemma}
\begin{proof} Let $x_{\xi_i}$ and $x_{\xi_j}$ be given by \eqref{x_i}, where $i\neq j$.
We consider two cases.\\
{\it Case $\xi_i=1$ and $\xi_j\neq1$.} Then from \eqref{x_i} one finds
\begin{eqnarray*}
x_{\xi_i}-x_{\xi_j}&=&\left(k-\frac{(k-1)q}{2}+\frac{(k-1)(k-2)q^2}{6k}-\frac{q}{1-\xi_j}\right)(\theta-1)\\
&=&(k+o[k])(\theta-1)\\
&=&O[k(\theta-1)],
\end{eqnarray*}
which implies that $|x_{\xi_i}-x_{\xi_j}|_p>|q(\theta-1)|_p$. Hence, $B_r({x}_{\xi_i})\cap B_r({x}_{\xi_j})=\emptyset$.\\
{\it Case $\xi_i\neq1$ and $\xi_j\neq1$.} In this case, we have
\begin{eqnarray*}
x_{\xi_i}-x_{\xi_j}&=&\frac{(\xi_i-\xi_j)q(\theta-1)}{(1-\xi_i)(1-\xi_j)}\\
&=&\frac{O[1]q(\theta-1)}{O[1]}\\
&=&O[q(\theta-1)],
\end{eqnarray*}
this means $|x_{\xi_i}-x_{\xi_j}|_p=|q(\theta-1)|_p$. Since $r=|q(\theta-1)|_p$, we infer that
$B_r({x}_{\xi_i})\cap B_r({x}_{\xi_j})=\emptyset$.
\end{proof}

\begin{pro}\label{prop_asosiy}
Let $p\geq3$ and $|k|_p>|q|_p$. If $|\theta-1|_p<|q^2|_p$ then
$$
\lim\limits_{n\to\infty}f_{\theta,q,k}^n(x)=1, \ \ \forall x\in Dom(f_{\theta,q,k})\setminus X.
$$
\end{pro}
\begin{proof} Let $x\in Dom(f_{\theta,q,k})\setminus X$.
According to Proposition \ref{pro_limits} it is enough to establish
$f_{\theta,q,k}^n(x)\to1$ for $|x-2+q+\theta|_p\leq|\theta-1|_p$.
So, we need to consider only a case when $x\in Dom(f_{\theta,q,k})\setminus X$ has the following form:
$$
x=2-q-\theta+\eta(\theta-1),
$$
where $|\eta|_p\leq1$. Then we immediately find
\begin{equation}\label{yordam}
g_{\theta,q}(x)-1=\frac{-q+(\eta-1)(\theta-1)}{\eta}.
\end{equation}
We consider several cases w.r.t $\eta$.

{\bf Case $|\eta|_p<|q|_p$.}  In this case
we have $|g_{\theta,q}(x)-1|_p>1$ which implies $|g_{\theta,q}(x)|_p>1$. Hence, $|f_{\theta,q,k}(x)|_p>1$. From the last
inequality using the strong triangle inequality we find
$|f_{\theta,q,k}(x)-1+q|_p>|q|_p$. Then
due to Corollary \ref{corssilka} we infer that $f_{\theta,q,k}^n(x)\to1$.

{\bf Case $|\eta|_p=|q|_p$.} Then there exists $\eta_*\in\mathbb Z_p^*$ such that
$\eta=q\eta_*$. Since $x\not\in X$ one gets
\begin{equation}\label{f1}
\left(1-\frac{1}{\eta_*}\right)^k-1=O[1].
\end{equation}
From \eqref{yordam} together with $|\theta-1|_p<|q|_p$ we infer
$$
g_{\theta,q}(x)=1-\frac{1}{\eta_*}+o[1].
$$
Hence, 
$$
f_{\theta,q,k}(x)=\left(1-\frac{1}{\eta_*}\right)^k+o[1].
$$
The last one together with \eqref{f1} yields $|f_{\theta,q,k}(x)-1|_p=1$. Again using the strong
triangle inequality we obtain $|f_{\theta,q,k}(x)-1+q|_p>|q|_p$. Consequently,
from Corollary \ref{corssilka} it follows  $f_\theta^n(x)\to1$.\\

{{\bf Case} $|q|_p<|\eta|_p\leq1$.} Note that $|q|_p<|\eta|_p$ implies $x\notin B_r(x_{\xi_i})$ for any $\xi_i^k=1$,
$\xi_i\neq1$. Therefore, from $x\notin X$ one find  
either $|k|_p\leq|\eta-k|_p\leq1$ or
\begin{equation}\label{f2}
|q|_p\leq\left|\eta-k+\frac{(k-1)q}{2}-\frac{(k-1)(k-2)q^2}{6k}\right|_p<|k|_p
\end{equation}
First, let us assume that $|k|_p\leq|\eta-k|_p\leq1$. Then from $|\theta-1|_p<|q^2|_p$ and $|q|_p<|\eta|_p$ one has
$$
g_{\theta,q}(x)=1-\frac{q}{\eta}+o[q],
$$
which yields
$$
f_{\theta,q,k}(x)=1-\frac{kq}{\eta}+o[q].
$$
The last equality implies $|f_{\theta,q,k}(x)-1+q|_p\geq|q|_p$. So, according
to Corollary \ref{corssilka} we have $f_{\theta,q,k}^n(x)\to1$.

Now we suppose that \eqref{f2} holds. Since $|q|_p<|k|_p$, using the strong triangle inequality from \eqref{f2}
we obtain $|\eta-k|_p<|k|_p$.
Furthermore, from \eqref{yordam} it follows that
\begin{eqnarray}\label{f4}
f_{\theta,q,k}(x)&=&\left(1-\frac{q}{\eta}\right)^k+O\left[\frac{\theta-1}{\eta}\right]\nonumber\\[2mm]
&=&\left(1-\frac{q}{k}\sum_{n=0}^\infty\left(\frac{k-\eta}{k}\right)^n\right)^k+O\left[\frac{\theta-1}{\eta}\right]\nonumber\\[2mm]
&=&1-q\sum_{n=0}^\infty\left(\frac{k-\eta}{k}\right)^n+\frac{(k-1)q^2}{2k}-\frac{(k-1)(k-2)q^3}{6k^2}+o\left[\frac{q^2}{k}\right]\nonumber\\[2mm]
&=&1-q+\frac{q}{k}\left(\eta-k+\frac{(k-1)q}{2}-\frac{(k-1)(k-2)q^2}{6k}\right)+q\sum_{n=2}^\infty\left(\frac{k-\eta}{k}\right)^n+o\left[\frac{q^2}{k}\right]
\end{eqnarray}
On the other hand, from \eqref{f2} one finds
\begin{equation}\label{ghty}
\left|\eta-k+\frac{(k-1)q}{2}-\frac{(k-1)(k-2)q^2}{6k}\right|_p=\left\{
\begin{array}{ll}
|\eta-k|_p, & \mbox{if } |\eta-k|_p>|q|_p,\\
|q|_p, & \mbox{if } |\eta-k|_p\leq|q|_p.
\end{array}\right.
\end{equation}
The last equality together with $|\eta-k|_p<|k|_p$, $|q|_p<|k|_p$ yields
\begin{equation}\label{f3}
\frac{\left|(k-\eta)^2\right|_p}{\left|k^2\right|_p}<\frac{\left|6k(\eta-k)+3k(k-1)q-(k-1)(k-2)q^2\right|_p}{|6k^2|_p}
\end{equation}
Plugging \eqref{f2} and \eqref{f3} into \eqref{f4} one finds
$$
|f_{\theta,q,k}(x)-1+q|_p\geq\frac{|q^2|_p}{|k|_p},
$$
which with $|k|_p\leq1$, $|\theta-1|_p<|q^2|_p$ implies
$|f_{\theta,q,k}(x)-2+q+\theta|_p>|\theta-1|_p$. Hence, thanks to Proposition \ref{pro_limits} we obtain
that $f_{\theta,q,k}^n(x)\to1$ as $n\to\infty$. This completes the proof.
\end{proof}

Let us define
\begin{equation}\label{juliaset}
J_{f_{\theta,q,k}}=\bigcap\limits_{n=1}^\infty f_{\theta,q,k}^{-n}(X).
\end{equation}

\begin{proof}[Proof of Theorem \ref{thm_mainasos}: $(B)$] Due to Proposition \ref{prop_asosiy} the set $\mathcal P_{x^{(\infty)}}$ can not
belong to $Dom(f_{\theta,q,k})\setminus X$. Then we get $\mathcal P_{x^{(\infty)}}\subset X$. According
to the construction of $J_{f_{\theta,q,k}}$, we conclude that $J_{f_{\theta,q,k}}\cap\mathcal P_{x^{(\infty)}}=\emptyset$.
On the other hand, $J_{f_{\theta,q,k}}$ is an invariant w.r.t. $f_{\theta,q,k}$. Then for any
$x\not\in J_{f_{\theta,q,k}}\cup\mathcal P_{x^{(\infty)}}$ there exists a number $m\geq1$ such that
$f_{\theta,q,k}^m(x)\not\in X$. Hence, due to Proposition
\ref{prop_asosiy} we infer that $f_{\theta,q,k}^n(x)\to1$ as $n\to\infty$.
The proof is complete.
\end{proof}

\section{Proof of Theorem \ref{thm_mainasos}: parts $(B1)$ and $(B2)$}

In the sequel, we need some auxiliary facts.

\begin{lemma}\label{lem_yanadayaxshi}
Let $p\geq3$ and $|k|_p>|q|_p$. Then for any $a\in B_{|q^2|_p}(1-q)$ the equation
$x^k=a$ has a unique solution $x_*$ on $\mathcal E_p$. Moreover, this solution
satisfies
\begin{equation}\label{x_*}
x_*-1+\frac{q}{k}+\frac{(k-1)q^2}{2k^2}-\frac{(k-1)(k-2)q^3}{6k^3}=o\left[\frac{q^2}{k^2}\right].
\end{equation}
\end{lemma}

\begin{proof} Let $|k|_p>|q|_p$ and $a\in B_{|q^2|_p}(1-q)$.
For convenience, we use the canonical form of $a$:
$$
a=1+a_tp^t+a_{t+1}p^{t+1}+\dots
$$
We note that $|k|_p>p^{-t}$.  Let us put  $x_t=1$ and define a sequence
$\{x_{n+t-1}\}_{n\geq1}$ as follows
\begin{equation}\label{p1p1}
x_{n+t}=x_{n+t-1}+\frac{a-x_{n+t-1}^k}{k}.
\end{equation}
First, by induction, let us show that $x_{n+t-1}\in\mathcal E_p$ for any $n\geq1$. It is clear that $x_t\in\mathcal E_p$, and therefore,
we assume that $x_{n+t-1}\in\mathcal E_p$ for some $n\geq1$. Then due to Corollary
\ref{coralb} we obtain
$$
x_{n+t-1}^k-1=k(x_{n+t-1}-1)+o[k(x_{n+t-1}-1)],
$$
which is equivalent to
$$
\left|x_{n+t-1}^k-1\right|_p<|k|_p.
$$
The last inequality together with $|a-1|_p<|k|_p$ implies that $|x_{n+t}-x_{n+t-1}|_p<1$.
Consequently, from $x_{n+t-1}\in\mathcal E_p$ we find $x_{n+t}\in\mathcal E_p$. So, we conclude that
$x_{n+t}\in\mathcal E_p$ for any $n\geq1$.

Due to Corollary \ref{coralb}, from \eqref{p1p1} we have
\begin{eqnarray*}
x_{n+t}^k-x_{n+t-1}^k&=&k(x_{n+t}-x_{n+t-1})+o[k(x_{n+t}-x_{n+t-1})]\\
&=&a-x_{n+t-1}^k+o[a-x_{n+t-1}],
\end{eqnarray*}
which means
$$
\left|x_{n+t}^k-a\right|_p<\left|x_{n+t-1}^k-a\right|_p.
$$
Hence, there exists a number $n_0\geq1$ such that
$$\left|x_{n_0+t}^k-a\right|_p\leq|(a-1)^2|_p.$$ Consider a polynomial $F(x)=x^k-a$.
It is easy to check that 
$$
|F'(x_{n_0+t-1})|_p=|k|_p, \ \ \textrm{and} \ \  |F(x_{n_0+t-1})|_p\leq|(a-1)^2|_p.$$
So, from $|k^2|_p>|(a-1)^2|_p$, thanks
to
Hensel's Lemma we conclude that $F$ has a solution $x_*$ such that
$$|x_*-x_{n_0+t-1}|_p\leq|(a-1)^2|_p.$$ From $x_{n_0+t-1}\in\mathcal E_p$ we infer
that $x_*\in\mathcal E_p$. The uniqueness of solution on $\mathcal E_p$ immediately
follows from Corollary \ref{coralb}.

Suppose that $x_*\in\mathcal E_p$ is a solution of $x^k-a=0$. Let us show that it can be represented by \eqref{x_*}.
Due to Corollary
\ref{coralb} the solution $x_*$ has the following form
\begin{equation}\label{x_*al}
x_*=1-\frac{q}{k}+\alpha_*
\end{equation}
where $\alpha_*=o\left[\frac{q}{k}\right]$.
Furthermore,
\begin{eqnarray*}
a=x_*^k&=&1+k\left(-\frac{q}{k}+\alpha_*\right)+\frac{k(k-1)}{2}\left(-\frac{q}{k}+\alpha_*\right)^2\\[2mm]
&&+\frac{k(k-1)(k-2)}{6}
\left(-\frac{q}{k}+\alpha_*\right)^3+o\left[\frac{q^2}{k}\right]\\[2mm]
&=&1-q+k\alpha_*+\frac{(k-1)q^2}{2k}-\frac{(k-1)(k-2)q^3}{6k^2}+o\left[\frac{q^2}{k}\right]\\[2mm]
\end{eqnarray*}
From $a=1-q+o[q^2]$ we obtain
$$
k\alpha_*+\frac{(k-1)q^2}{2k}-\frac{(k-1)(k-2)q^3}{6k^2}=o\left[\frac{q^2}{k}\right]
$$
which implies that
$$
\alpha_*=-\frac{(k-1)q^2}{2k^2}+\frac{(k-1)(k-2)q^3}{6k^3}+o\left[\frac{q^2}{k^2}\right]
$$
Putting the last one into \eqref{x_*al} one gets \eqref{x_*}, which completes the proof.
\end{proof}

\begin{rk}
Thanks to Lemma \ref{lem_yanadayaxshi}, every $a\in\mathcal E_p$ with $|a-1|_p<|k|_p$, the equation $x^k=a$
has a single root belonging to $\mathcal E_p$, which is called {\it the principal $k$-th root}
and it is denoted by $\sqrt[k]{a}$. In what follows, when we write $\sqrt[k]{a}$ for given
$a\in\mathcal E_p$
we always mean the principal $k$-th root of $a$.
\end{rk}

\begin{rk}
We point out that in \cite{MS13} the existence (only) of solutions of the equation $x^k=a$ on $\mathbb Z^*_p$ has been obtained, but an advantage of  Lemma \ref{lem_yanadayaxshi} is that it provides uniqueness of the solution in $\ce_p$ with its explicit  expression which is essential in our investigation.
\end{rk}

On the set $X$ (see \eqref{X}) the mapping $f_{\theta,q,k}$ has exactly $\kappa_p$ inverse branches:
$$
h_{i}(x)=\frac{(q+\theta-2)\xi_i\sqrt[k]{x}-q+1}{\theta-\xi_i\sqrt[k]{x}},\ \ \xi_i^k=1,\ \ i\in\{1,\dots,\kappa_p\}.
$$

\begin{pro}\label{pro_teskari} Let $p\geq3$ and $|k|_p>|q|_p$. If $|\theta-1|_p<|q^2|_p$ then
\begin{equation}\label{t1}
h_{i}(X)\subset B_r(x_{\xi_i}).
\end{equation}
\end{pro}

\begin{proof} Let $x\in X$. We consider two cases: $\xi_i=1$ and $\xi_i\neq1$.\\
{\bf Case $\xi_i=1$.} In this case, we have
\begin{eqnarray}
h_{i}(x)-x_{\xi_i}&=&\frac{(q+\theta-2)\sqrt[k]{x}-q+1}{\theta-\sqrt[k]{x}}-
\left(1-q+(k-1)\left(1-\frac{q}{2}+\frac{(k-2)q^2}{6k}\right)(\theta-1)\right)\nonumber\\[2mm]
&=&\frac{(\theta-1)\left(q+\theta-1+\left(k-\frac{(k-1)q}{2}+\frac{(k-1)(k-2)q^2}{6k}\right)(\sqrt[k]{x}-\theta)\right)}{\theta-\sqrt[k]{x}}\nonumber\\[2mm]
\label{t4}&=&\frac{(\theta-1)\left(q+\left(k-\frac{(k-1)q}{2}+\frac{(k-1)(k-2)q^2}{6k}\right)(\sqrt[k]{x}-1)+o[q^2]\right)}{1-\sqrt[k]{x}+o[q^2]}
\end{eqnarray}
On the other hand, due to Lemma \ref{lem_yanadayaxshi} we have
\begin{equation}\label{t2}
\sqrt[k]{x}-1=-\frac{q}{k}-\frac{(k-1)q^2}{2k^2}+\frac{(k-1)(k-2)q^3}{6k^3}+o\left[\frac{q^2}{k^2}\right]
\end{equation}
Furthermore, keeping in mind $|q|_p<|k|_p$, we can easily check the following
\begin{equation}\label{t3}
q+\left(k-\frac{(k-1)q}{2}+\frac{(k-1)(k-2)q^2}{6k}\right)(\sqrt[k]{x}-1)=o\left[\frac{q^2}{k}\right]
\end{equation}
Plugging \eqref{t2},\eqref{t3} into \eqref{t4} one has
\begin{eqnarray*}
h_{i}(x)-x_{\xi_i}&=&
\frac{(\theta-1)o\left[\frac{q^2}{k}\right]}{O\left[\frac{q}{k}\right]}\\[2mm]
&=&(\theta-1)o[q]\\[2mm]
&=&o[q(\theta-1)].
\end{eqnarray*}
This means $h_{i}(x)\in B_r(x_{\xi_i})$. Arbitrariness of
$x\in X$ yields \eqref{t1}.\\
{\bf Case $\xi_i\neq1$.} Then, we have
\begin{eqnarray*}
h_{i}(x)-x_{\xi_i}&=&\frac{(q+\theta-2)\xi_i\sqrt[k]{x}-q+1}{\theta-\xi_i\sqrt[k]{x}}-\left(2-q-\theta+\frac{q}{1-\xi_i}(\theta-1)\right)\\[2mm]
&=&\frac{(\theta-1)\left(q-\frac{\theta q}{1-\xi_i}+\frac{\xi_i\sqrt[k]{x}q}{1-\xi_i}+\theta-1\right)}{\theta-1+1-\xi_i\sqrt{x}}\\[2mm]
&=&\frac{(\theta-1)\left(q-\frac{q}{1-\xi_i}+\frac{\xi_iq}{1-\xi_i}+o[q]\right)}{O[1]}\\[2mm]
&=&\frac{(\theta-1)o[q]}{O[1]}\\[2mm]
&=&o[q(\theta-1)]
\end{eqnarray*}
The last one implies $h_{i}(x)\in B_r(x_{\xi_i})$. Again due to thenarbitrariness
of $x\in X$ we obtain \eqref{t1}.
\end{proof}

\begin{cor}\label{cor_1234}
Let $p\geq3$ and $|k|_p>|q|_p$. If $|\theta-1|_p<|q^2|_p<1$ and $X$ is a set given by \eqref{X} then
the following statements hold:
\begin{enumerate}
\item[$(i)$] $f_{\theta,q,k}^{-1}(X)\subset X$;\\
\item[$(ii)$] $B_r({x}_{\xi_i})\subset f_{\theta,q,k}(B_{r}(x_{\xi_j}))$ for any $i,j\in\{1,\dots,\kappa_p\}$.
\end{enumerate}
\end{cor}

\begin{pro}\label{pro_vashshe}
Let $p\geq3$, $|k|_p>|q|_p$. If $|\theta-1|_p<|q^2|_p$ and $X$ be set given by \eqref{X} then
the following statements hold:\\
\begin{enumerate}
\item[$(a)$] if $\xi_i=1$ then
\begin{equation}\label{fx-fy1}
|f_{\theta,q,k}(x)-f_{\theta,q,k}(y)|_p=\frac{|q(x-y)|_p}{|k(\theta-1)|_p},\
\ \ \mbox{for any }x,y\in B_r(x_{\xi_i}).
\end{equation}
\item[$(b)$] if $\xi_i\neq1$ then
\begin{equation}\label{fx-fy2}
|f_{\theta,q,k}(x)-f_{\theta,q,k}(y)|_p=\frac{|k(x-y)|_p}{|q(\theta-1)|_p},\
\ \ \mbox{for any }x,y\in B_r(x_{\xi_i}).
\end{equation}
\end{enumerate}
\end{pro}
\begin{proof}
$(a)$ Let $x\in B_r(x_{\xi_i})$, where
$$x_{\xi_i}=1-q+(k-1)\left(1-\frac{q}{2}+\frac{(k-2)q^2}{6k}\right)(\theta-1).$$

Then we have
\begin{eqnarray*}
g_{\theta,q}(x)-1&=&\frac{(\theta-1)(x-1)}{x-2+q+\theta}\\[2mm]
&=&\frac{(\theta-1)(-q+o[q])}{k(\theta-1)+o[k(\theta-1)]}\\[2mm]
&=&O\left[\frac{q}{k}\right]
\end{eqnarray*}
this means $g_{\theta,q}(x)\in\mathcal E_p$. Then for any $x,y\in B_r(x_{\xi_i})$
due to Corollary \ref{coralb} one gets
\begin{equation}\label{y1}
|f_{\theta,q,k}(x)-f_{\theta,q,k}(y)|_p=|k(g_{\theta,q}(x)-g_{\theta,q}(y))|_p
\end{equation}
On the other hand, we have
\begin{eqnarray*}
g_{\theta,q}(x)-g_{\theta,q}(y)&=&\frac{(\theta-1)(q+\theta-1)(x-y)}{(x-2+q+\theta)(y-2+q+\theta)}\\[2mm]
&=&\frac{O[q(\theta-1)](x-y)}{\left(k(\theta-1)+o[k(\theta-1)]\right)^2}\\[2mm]
&=&\frac{O[q](x-y)}{O[k^2(\theta-1)]}
\end{eqnarray*}
Putting the last one into \eqref{y1} one finds \eqref{fx-fy1}.\\
$(b)$ Let $x\in B_r(x_{\xi_i})$, where $x_{\xi_i}=2-q+\theta+\frac{q(\theta-1)}{1-\xi_i}$.
Then 
\begin{eqnarray*}
g_{\theta,q}(x)&=&1+\frac{(\theta-1)(x-1)}{x-2+q+\theta}\\[2mm]
&=&1+\frac{(\theta-1)(-q+o[q])}{\frac{q(\theta-1)}{1-\xi_i}+o[q(\theta-1)]}\\[2mm]
&=&\xi_i+o[1]
\end{eqnarray*}
so,  $|g_{\theta,q}(x)|_p=1$. Moreover, $\frac{g_{\theta,q}(x)}{g_{\theta,q}(y)}\in\mathcal E_p$ for any
$x,y\in B_r(x_{\xi_i})$. Then due to Corollary \ref{coralb} we  have
\begin{equation}\label{y2}
|f_{\theta,q,k}(x)-f_{\theta,q,k}(y)|_p=|k(g_{\theta,q}(x)-g_{\theta,q}(y))|_p
\end{equation}
On the other hand, one finds
\begin{eqnarray*}
g_{\theta,q}(x)-g_{\theta,q}(y)&=&\frac{(\theta-1)(q+\theta-1)(x-y)}{(x-2+q+\theta)(y-2+q+\theta)}\\[2mm]
&=&\frac{O[q(\theta-1)](x-y)}{\left(\frac{q(\theta-1)}{1-\xi_i}+o[q(\theta-1)]\right)^2}\\[2mm]
&=&\frac{(x-y)}{O[q(\theta-1)]}
\end{eqnarray*}
Plugging this into \eqref{y2} we obtain \eqref{fx-fy2}.
This completes the proof.
\end{proof}

\begin{proof}[Proof of Theorem \ref{thm_mainasos}] $(B1)$
Assume that $x^k=1$ has only one solution. Then the set $X$ given by \eqref{X} consists of only
one ball $B_r(x_1)$, where
$$
x_1=1-q+(k-1)(1-\frac{q}{2})(\theta-1),\ \ \ r=|q(\theta-1)|_p.
$$
Thanks to the case $(B)$ we have
$$
A(x_0^*)=Dom(f_{\theta,q,k})\setminus(J_{f_{\theta,q,k}}\cup\mathcal P_{x^{(\infty)}}),
$$
where $x_0^*=1$ and $J_{f_{\theta,q,k}}$ is given by \eqref{juliaset}.
Due to Proposition \ref{pro_vashshe} for any $x,y\in B_r(x_1)$ we find
$$
|f_{\theta,q,k}(x)-f_{\theta,q,k}(y)|_p>p^2|x-y|_p,
$$
which implies that $|J_{f_{\theta,q,k}}|\leq1$. So, it is enough to show
$J_{f_{\theta,q,k}}\neq\emptyset$.

Let us consider the following polynomial
\begin{equation}\label{Fpol}
F(x)=x^k-1+q+(\theta-1)\sum_{j=1}^{k-1}x^{k-j}
\end{equation}
Since $|k|_p>|q|_p$ and $|\theta-1|_p<|q^2|_p$ due to the same argument as Lemma
\ref{lem_yanadayaxshi} the polynomial $F$ has a unique root
$x_{**}$ in $\mathcal E_p$. Furthermore, we obtain
$$
x_{**}^k-2+q+\theta=(\theta-1)\sum_{j=1}^kx_{**}^{k-j},
$$
which yields
$$
\frac{(\theta-1)(x_{**}^k-1)}{x_{**}^k-2+q+\theta}=x_{**}-1.
$$
Denoting $x_*=x_{**}^k$, from the last one, we find
$x_*\in\mathcal E_p\setminus\{1\}$ and $f_{\theta,q,k}(x_*)=x_*$. This means that
$x_*\in J_{f_{\theta,q,k}}$. From  $|J_{f_{\theta,q,k}}|\leq1$ one has $J_{f_{\theta,q,k}}=\{x_*\}$.

$(B2).$ Assume that $x^k=1$ has $\kappa_p$ ($\kappa_p\geq2$) solutions. Consider the set $X$ defined by \eqref{X}.
Then according to Corollary \ref{cor_1234} $(i)$ and Proposition \ref{pro_vashshe} the triple
$(X,J_{f_{\theta,q,k}},f_{\theta,q,k})$ is a $p$-adic repeller. Due to Corollary \ref{cor_1234} $(ii)$, the corresponding incidence matrix $A$, 
has dimension $\kappa_p\times\kappa_p$ and can be written as follows 
$$
A=\left(\begin{array}{llll}
1 & 1 & \dots & 1\\
1 & 1 & \dots & 1\\
\vdots & \vdots & \ddots & \vdots\\
1 & 1 & \dots & 1
\end{array}
\right)
$$
This means that a triple $(X,J_{f_{\theta,q,k}},f_{\theta,q,k})$ is
transitive, hence Theorem \ref{xit} implies  that
the dynamics $(J_{f_{\theta,q,k}},f_{\theta,q,k},|\cdot|_p)$ is topologically conjugate
to the full shift dynamics of $\kappa_p$ symbols.

This completes the proof.
\end{proof}

\appendix

\section{$p$-adic measure}

Let $(X,\cb)$ be a measurable space, where $\cb$ is an algebra of
subsets $X$. A function $\m:\cb\to \bq_p$ is said to be a {\it
$p$-adic measure} if for any $A_1,\dots,A_n\subset\cb$ such that
$A_i\cap A_j=\emptyset$ ($i\neq j$) the equality holds
$$
\mu\bigg(\bigcup_{j=1}^{n} A_j\bigg)=\sum_{j=1}^{n}\mu(A_j).
$$

A $p$-adic measure is called a {\it probability measure} if
$\mu(X)=1$.  A $p$-adic probability measure $\m$ is called {\it
bounded} if $\sup\{|\m(A)|_p : A\in \cb\}<\infty $. For more detail
information about $p$-adic measures we refer to
\cite{K3},\cite{KhN}.

\section{Cayley tree}

Let $\Gamma^k_+ = (V,L)$ be a semi-infinite Cayley tree of order
$k\geq 1$ with the root $x^0$ (whose each vertex has exactly $k+1$
edges, except for the root $x^0$, which has $k$ edges). Here $V$ is
the set of vertices and $L$ is the set of edges. The vertices $x$
and $y$ are called {\it nearest neighbors} and they are denoted by
$l=<x,y>$ if there exists an edge connecting them. A collection of
the pairs $<x,x_1>,\dots,<x_{d-1},y>$ is called a {\it path} from
the point $x$ to the point $y$. The distance $d(x,y), x,y\in V$, on
the Cayley tree, is the length of the shortest path from $x$ to $y$.
$$
W_{n}=\left\{ x\in V\mid d(x,x^{0})=n\right\}, \ \
V_n=\overset{n}{\underset{m=0}{\bigcup}}W_{m}, \ \ L_{n}=\left\{
l=<x,y>\in L\mid x,y\in V_{n}\right\}.
$$
The set of direct successors of $x$ is defined by
$$
S(x)=\left\{ y\in W_{n+1}:d(x,y)=1\right\}, x\in W_{n}.
$$
Observe that any vertex $x\neq x^{0}$ has $k$ direct successors and
$x^{0}$ has $k+1$.

\section{$p$-adic quasi Gibbs measure}

In this section we recall the definition of  $p$-adic quasi Gibbs
measure (see \cite{M12}).

Let $\Phi=\{1,2,\cdots,q\}$, here $q\geq 2$, ($\Phi$ is called a
{\it state space}) and is assigned to the vertices of the tree
$\G^k_+=(V,\Lambda)$. A configuration $\s$ on $V$ is then defined as
a function $x\in V\to\s(x)\in\Phi$; in a similar manner one defines
configurations $\s_n$ and $\w$ on $V_n$ and $W_n$, respectively. The
set of all configurations on $V$ (resp. $V_n$, $W_n$) coincides with
$\Omega=\Phi^{V}$ (resp. $\Omega_{V_n}=\Phi^{V_n},\ \
\Omega_{W_n}=\Phi^{W_n}$). One can see that
$\Om_{V_n}=\Om_{V_{n-1}}\times\Om_{W_n}$. Using this, for given
configurations $\s_{n-1}\in\Om_{V_{n-1}}$ and $\w\in\Om_{W_{n}}$ we
define their concatenations  by
$$
(\s_{n-1}\vee\w)(x)= \left\{
\begin{array}{ll}
\s_{n-1}(x), \ \ \textrm{if} \ \  x\in V_{n-1},\\
\w(x), \ \ \ \ \ \ \textrm{if} \ \ x\in W_n.\\
\end{array}
\right.
$$
It is clear that $\s_{n-1}\vee\w\in \Om_{V_n}$.

The (formal) Hamiltonian of $p$-adic Potts model is
\begin{equation}\label{ph}
H(\sigma)=J\sum_{\langle x,y\rangle\in L}
\delta_{\sigma(x)\sigma(y)},
\end{equation}
where $J\in B(0, p^{-1/(p-1)})$ is a coupling constant, and
$\delta_{ij}$ is the Kroneker's symbol.

A construct of a  generalized $p$-adic quasi Gibbs measure
corresponding to the model is given below.

Assume that  $\h: V\setminus\{x^{(0)}\}\to\bq_p^{\Phi}$ is a
mapping, i.e. $\h_x=(h_{1,x},h_{1,x},\dots,h_{q,x})$, where
$h_{i,x}\in\bq_p$ ($i\in\Phi$) and $x\in V\setminus\{x^{(0)}\}$.
Given $n\in\bn$, we consider a $p$-adic probability measure
$\m^{(n)}_{\h,\rho}$ on $\Om_{V_n}$ defined by
\begin{equation}\label{mu}
\mu^{(n)}_{\h}(\s)=\frac{1}{Z_{n}^{(\h)}}\exp\{H_n(\s)\}\prod_{x\in
W_n}h_{\s(x),x}
\end{equation}
Here, $\s\in\Om_{V_n}$, and $Z_{n}^{(\h)}$ is the corresponding
normalizing factor
\begin{equation}\label{ZN1}
Z_{n}^{(\h)}=\sum_{\s\in\Omega_{V_n}}\exp\{H_n(\s)\}\prod_{x\in
W_n}h_{\s(x),x}.
\end{equation}

In this paper, we are interested in a construction of an infinite
volume distribution with given finite-dimensional distributions.
More exactly, we would like to find a $p$-adic probability measure
$\m$ on $\Om$ which is compatible with given ones $\m_{\h}^{(n)}$,
i.e.
\begin{equation}\label{CM}
\m(\s\in\Om: \s|_{V_n}=\s_n)=\m^{(n)}_{\h}(\s_n), \ \ \ \textrm{for
all} \ \ \s_n\in\Om_{V_n}, \ n\in\bn.
\end{equation}

We say that the $p$-adic probability distributions \eqref{mu} are
\textit{compatible} if for all $n\geq 1$ and $\sigma\in
\Phi^{V_{n-1}}$:
\begin{equation}\label{comp}
\sum_{\w\in\Om_{W_n}}\m^{(n)}_{\h}(\s_{n-1}\vee\w)=\m^{(n-1)}_{\h}(\s_{n-1}).
\end{equation}
 This condition according to the Kolmogorov extension theorem (see \cite{KL}) implies the existence of a unique $p$-adic measure
$\m_{\h}$ defined on $\Om$ with a required condition \eqref{CM}.
Such a measure $\m_{\h}$ is said to be {\it a $p$-adic quasi Gibbs
measure} corresponding to the model \cite{M12,M13}. If one has
$h_x\in\ce_p$ for all $x\in V\setminus\{x^{(0)}\}$, then the
corresponding measure $\m_\h$ is called \textit{$p$-adic Gibbs
measure} (see \cite{MR1}).

By $Q\cg(H)$ we denote the set of all $p$-adic quasi Gibbs measures
associated with functions $\h=\{\h_x,\ x\in V\}$. If there are at
least two distinct generalized $p$-adic quasi Gibbs measures such
that at least one of them is unbounded, then we say that \textit{a
phase transition} occurs.

The following statement describes conditions on $h_x$ guaranteeing
compatibility of $\mu_{\bf h}^{(n)}(\sigma)$.

\begin{thm}\label{comp1}\cite{M12} The measures $\m^{(n)}_{\h}$, $
n=1,2,\dots$ (see \eqref{mu}) associated with $q$-state Potts model
\eqref{ph} satisfy the compatibility condition \eqref{comp} if and
only if for any $n\in \bn$ the following equation holds:
\begin{equation}\label{eq1}
\hat h_{x}=\prod_{y\in S(x)}{\mathbf{F}}(\hat \h_{y},\theta),
\end{equation}
here and below a vector $\hat \h=(\hat h_1,\dots,\hat
h_{q-1})\in\bq_p^{q-1}$ is defined by a vector
$\h=(h_1,h_1,\dots,h_{q})\in\bq_p^{q}$ as follows
\begin{equation}\label{H}
\hat h_i=\frac{h_i}{h_q}, \ \ \ i=1,2,\dots,q-1
\end{equation}
and mapping ${\mathbf{F}}:\bq_p^{q-1}\times\bq_p\to\bq_p^{q-1}$ is
defined by
${\mathbf{F}}(\xb;\theta)=(F_1(\xb;\theta),\dots,F_{q-1}(\xb;\theta))$
with
\begin{equation}\label{eq2}
F_i(\xb;\theta)=\frac{(\theta-1)x_i+\sum\limits_{j=1}^{q-1}x_j+1}
{\sum\limits_{j=1}^{q-1}x_j+\theta}, \ \ \xb=\{x_i\}\in\bq_p^{q-1},
\ \ i=1,2,\dots,q-1.
\end{equation}
\end{thm}

Let us first observe that the set
$(\underbrace{1,\dots,1,h}_m,1,\dots,1)$ ($m=1,\dots,q-1$) is
invariant for the equation \eqref{eq1}. Therefore, in what follows,
we restrict ourselves to one of such lines, let us say
$(h,1,\dots,1)$.

In \cite{MFKh16} to establish the phase transition, we considered
translation-invariant (i.e. $\h=\{\h_x\}_{x\in V\setminus\{x^0\}}$
such that $\h_x=\h_y$ for all $x,y$) solutions of \eqref{eq1}. Then
the equation \eqref{eq1} reduced to the following one
\begin{equation}\label{eq12}
h=f_\t(h),
\end{equation}
where
\begin{equation}\label{f}
f_\t(x)=\bigg(\frac{\t x+q-1}{x+\t+q-2}\bigg)^k.
\end{equation}

Hence, to establish the existence of the phase transition, when
$k=2$, we showed \cite{MR1} that \eqref{eq12} has three nontrivial
solutions if $q$ is divisible by $p$.  Note that full description of
all solutions of the last equation has been carried out in
\cite{RKh} when $k=2$.

\section*{Acknowledgments}
The present work is supported by the UAEU "Start-Up" Grant, No.
31S259.

\end{document}